\documentclass[12pt]{article}
\usepackage[a4paper, margin=1in]{geometry}
\usepackage[utf8]{inputenc}
\usepackage{braket}
\usepackage{tikz}
\usepackage{amssymb}
\usepackage{amsmath}
\usepackage{mathrsfs}
\usepackage{mathtools}
\usepackage{graphicx}
\usepackage{enumerate}

\usepackage{amsthm}
\newtheorem{theorem}{Theorem}
\newtheorem{lemma}{Lemma}

\newtheorem*{hypothesis}{Hypothesis}
\usepackage{lipsum}

\makeatletter
\newcommand{\subjclass}[2][2022]{%
  \let\@oldtitle\@title%
  \gdef\@title{\@oldtitle\footnotetext{#1 \emph{Mathematics subject classification.} #2}}%
}
\newcommand{\keywords}[1]{%
  \let\@@oldtitle\@title%
  \gdef\@title{\@@oldtitle\footnotetext{\emph{Key words and phrases.} #1.}}%
}
\makeatother

\title{The number of preimages of iterates of $\phi$ and $\sigma$}
\author{Agbolade Akande}
\date{\today}

\begin{document}

\maketitle

\section{Introduction}
This paper explores the preimages of iterates of Euler's totient function $\phi$ and the sum-of-divisors function $\sigma$. Define $N_a^k(n) = \#\{ m : a_k(m) = n \}$, where $a_k$ is the $k$-th iterate of the arithmetic function $a$. Our main theorem is an upper bound, for each fixed $k$, on the functions $N_{\phi}^{k}(n)$ and $N_{\sigma}^{k}(n)$. This confirms a conjecture of de Koninck and Katai\cite{Corollary}.\\
\\
We first recall what is known about $N_\sigma^{1}(n)$ and $N_\phi^{1}(n)$. Define \[L(n) =\exp(\log n \cdot \log \log \log n/\log \log n).\] Erd\H{o}s showed that for a certain absolute constant $c > 0$, and all large enough $n$,
$N_{\phi}^1(n) \le n/L(n)^{c}$ \cite{Erdos}. Carl Pomerance found a more optimal upper bound for the number of preimages of $\phi$ and $\sigma$ \cite{Pom}:
\renewcommand*{\thetheorem}{\Alph{theorem}}
\begin{theorem}[Pomerance]
As $n \rightarrow \infty$,
$$N^1_\phi(n) \le \frac{n}{L(n)^{1+o(1)}},$$
$$N^1_\sigma(n) \le \frac{n}{L(n)^{1+o(1)}}.$$
\end{theorem}
\noindent
The proof for $N^1_\sigma(n)$ is not given explicitly in \cite{Pom} but it follows the same reasoning for $N_\phi^1(n)$. (This observation is made explicitly in \cite{Common}). Erd\H{o}s' and Pomerance's proofs require bounds on the number of smooth shifted primes (shifted primes $p-1$ without large prime factors). Under a plausible conjecture on the distribution of these primes, Pomerance shows in \cite{Pom} that the bounds in Theorem A are optimal, in the sense that the exponent $1+o(1)$ cannot be replaced with $1-\epsilon$ for any fixed $\epsilon > 0$. We will later use this idea to formulate a conjecture for a lower bound. \\
\\
Our upper bounds on $N_\phi^{k}(n)$  and $N_{\sigma}^{k}(n)$ borrow ideas from work of Pollack and Vandehey \cite{Normal} and Pollack \cite{Factor}. The proof goes by induction on $k$, the input for the case $k=1$ being the results of Pomerance (though the weaker results of Erd\H{o}s would also suffice). Let $a \in \{\phi, \sigma\}$. Observe that if $a_k(m)=n$, then $a(m)$ belongs to the set $S = \{\ell: a_{k-1}(\ell) = n\}$ counted by $N_a^{k-1}(n)$. We can assume by our the induction hypothesis that we have a usable upper bound on $|S|$. We split $S$ into three separate pieces. In the first piece, we put those $m \in S$ with many large prime factors. In the second piece, we place those $m$ which do not have many large prime factors, but have many prime factors in total. The final piece contains the leftovers. We then estimate the size, under $a^{-1}$, of each of these parts; this is carried out in sections \ref{Proof of Theorem 1.1} and \ref{Remarks for lower bound} after establishing some technical preliminary estimates in Section \ref{Bounds proof}.\\
\\
For each positive integer $k$ and each real number $\beta$, define \[L_{k,\beta}(n) = \exp\left(\log n \cdot (\log \log \log n)^\beta/(\log \log n)^k \right).\]
\noindent
The end result is as follows:
\renewcommand*{\thetheorem}{\arabic{theorem}}
\setcounter{theorem}{0}
\begin{theorem}\label{New Pomerance upper bound}
Fix $a \in \{\phi,\sigma\}$. Fix a positive integer $k$ and $\beta < k-1$. As $n \rightarrow \infty$, 
$$N_a^k(n) \le \frac{n}{L_{k,\beta+1}(n)^{1+o(1)}}.$$

\end{theorem}
\noindent
Theorem \ref{New Pomerance upper bound} establishes a strong form of a conjecture of de Koninck and Katai \cite[Conjecture 2]{Corollary}, who had $c_k n/(\log{n})^{9}$ in place of $n/L_{k,\beta+1}(n)^{1+o(1)}$. As a consequence, we derive the following theorem, which was proved conditionally on their conjecture in \cite{Corollary}:
\renewcommand*{\thetheorem}{\Alph{theorem}}
\begin{theorem}
Let $e(y) = e^{2\pi iy}$, and let $\mathcal{M}_1$ be the set of all multiplicative arithmetic functions $f$ such that $\lvert f(n) \rvert = 1$ for all integers $n$. Fix $k \in \mathbb{N}$, then for almost all real numbers $\alpha \in (0,1)$
$$\sup \limits_{f \in \mathcal{M}_1} \frac{1}{x}\left \lvert \sum \limits_{n \le x} f(n) e(\alpha \phi_k(n)) \right \rvert \rightarrow 0$$
$$\sup \limits_{f \in \mathcal{M}_1} \frac{1}{x}\left \lvert \sum \limits_{n \le x} f(n) e(\alpha \sigma_k(n)) \right \rvert \rightarrow 0$$
as $x \rightarrow \infty$.
\end{theorem}\setcounter{theorem}{1}
\noindent
\textbf{Notation and conventions.} The letter $p$ is reserved throughout for primes. We use $\Omega(n)$ for the number of prime factors of $n$ counted with multiplicity and $\Omega_{>z}(n)$ for the number of prime factors of $n$ exceeding $z$, counted with multiplicity.\\
\\
The $z$-rough and $z$-smooth parts of a positive integer $m$ denotes the largest divisors of $m$ composed of primes exceeding $z$, and not exceeding $z$ respectively.\\ 
\\
We will use $\Gamma(z)$ to represent the gamma function defined by the improper integral:
$$\Gamma(z) = \int_0^\infty t^{z-1}e^{-t}dt.$$

\section{Bounds on some moment generating functions}\label{Bounds proof}
In this section, we will find an upper bound on some moment generating functions involving the total number of prime factors and the number of large prime factors of $a(n)$.\\
\\
The next lemma gives our estimate for large prime factors of $a(n)$.\\
\\
Fix $\eta \in (0,1)$ and let
$$z = \exp((\log \log x)^{1/2}),$$
$$A = (\log \log x)^{1-\eta}.$$
\begin{lemma}\label{Moment generating function A}
Fix $a\in \{\phi,\sigma\}$. As $x \rightarrow \infty$
$$\sum \limits_{n \le x}A^{\Omega_{>z}(a(n))} \le x\exp\left(O\left(\frac{\log x}{\exp((\log \log x)^{\eta/2})}\right)\right).$$
\end{lemma}
\begin{proof}
Let $c \in (1,2)$. By Rankin's trick \cite{Rankin},
$$\sum \limits_{n \le x}A^{\Omega_{>z}(a(n))} \le x^c\sum \limits_{n \le x}\frac{A^{\Omega_{>z}(a(n))}}{n^c} \le x^c \prod \limits_{p \le x}\left(1+\sum \limits_{k=1}^\infty\frac{A^{\Omega_{>z}(a(p^k))}}{p^{kc}} \right)$$
$$\le x^c\exp \left(\sum \limits_{p\le x}\sum \limits_{k=1}^\infty\frac{A^{\Omega_{>z}(a(p^k))}}{p^{kc}} \right).$$
Notice that 
$$\Omega_{>z}(a(p^k)) \le \frac{\log(a(p^k))}{\log z} \le \frac{\log(\sigma(p^k))}{\log z} = \frac{\log\left(\frac{p^{k+1} - 1}{p-1}\right)}{\log z} \le \frac{(k+1)\log p - \log(p-1)}{\log z}.$$
Applying this back to the original sum,
$$ \sum \limits_{n \le x}A^{\Omega_{>z}(a(n))} \le x^c\exp\left(\sum \limits_{2\le p\le x}\left[ \frac{A^{\Omega_{>z}(a(p))}}{p^c}+ \frac{1}{A^{\frac{\log(p-1)}{\log z}}} \sum \limits_{k\ge 2}\frac{A^{\frac{(k+1)\log p}{\log z}}}{p^{kc}} \right]\right).$$
For $x$ large, $A < z^{1/3}$ so that
$$A^{\frac{\log p}{\log z}}/p^{c} \le z^{\frac{\log p}{3\log z}}/p^{c} \le p^{1/3-c} <p^{-2/3} \le 2^{-2/3}.$$
Therefore,
$$\sum \limits_{k\ge 2}\frac{A^{\frac{(k+1)\log p}{\log z}}}{p^{kc}} \ll \frac{A^{\frac{3\log p}{\log z}}}{p^{2c}}.$$
Thus,
$$\sum \limits_{n \le x}A^{\Omega_{>z}(a(n))} \le x \exp\left((c-1)\log x + O\left(\sum \limits_{2\le p \le x} \frac{A^{\Omega_{>z}(a(p))}}{p^c} + \sum \limits_{2\le p \le x}\frac{A^{\frac{3\log p - \log(p-1)}{\log z}}}{p^{2c}} \right) \right).$$
Notice that $\frac{A^{\frac{3\log p - \log(p-1)}{\log z}}}{p^{2c}} = \frac{\left(\frac{p^3}{p-1}\right)^{\frac{\log A}{\log z}}}{p^{2c}} \ll \frac{p^{\frac{2\log A}{\log z}}}{p^{2c}} \le p^{-4/3}$. Hence, $\sum \limits_{2\le p \le x}\frac{A^{\frac{3\log p - \log(p-1)}{\log z}}}{p^{2c}} \le \sum \limits_{2\le p \le x}p^{-4/3} = O(1) $.\\
\\
Thus, we need to bound $\sum \limits_{p \le x}\frac{A^{\Omega_{>z}(a(p))}}{p^c}$. Let $g$ be the arithmetic function defined implicitly by $A^{\Omega_{>z}(d)} = \sum \limits_{r \mid d}g(r)$. Then $g$ is multiplicative and $g(p^e) = A^e - A^{e-1}$ when $p > z$ and $g(p^e) = 0$ otherwise. Hence,
$$\sum \limits_{p \le T}A^{\Omega_{>z}(a(p))} = \sum \limits_{r \le T}g(r)\sum \limits_{\underset{r \mid a(p)}{p\le T}}1.$$
Since $a \in \left\{\phi,\sigma \right\}$, $\sum \limits_{\underset{r \mid a(p)}{p\le T}}1 \le \frac{(T+1)}{r}$. Hence,
$$\sum \limits_{p \le T}A^{\Omega_{>z}(a(p))} \le (T+1)\sum \limits_{r \le T}\frac{g(r)}{r}= (T+1)\prod \limits_{z<p \le T}\left(1+\frac{A-1}{p}+ \frac{A^2 - A}{p^2} +\cdots\right)$$
$$\le (T+1)\exp\left((A-1) \sum \limits_{z<p\le T}\frac{1}{p} \right)\exp\left(\sum \limits_{z<p\le T}\frac{A^2-A}{p^2} + \frac{A^3-A^2}{p^3} + \cdots \right).$$
The geometric ratio of successive terms in the second summation is $A/p \le A/z < \frac{1}{2}$. Hence,
$$\sum \limits_{z<p\le T}\left(\frac{A^2-A}{p^2} + \frac{A^3-A^2}{p^3} + \cdots \right) < 2\sum \limits_{z<p\le T}\frac{A^2-A}{p^2}\le 2A^2\sum \limits_{z<p\le T}\frac{1}{p^2}\ll A^2/z \ll 1.$$
By the prime number theorem,
$$\sum \limits_{z<p\le T}\frac{1}{p} = \log \frac{\log T}{\log z} + O\left(\frac{1}{(\log z)^2} \right)$$
so that
$$(A-1)\sum \limits_{z<p\le T}\frac{1}{p} = (A-1)\log \frac{\log T}{\log z} + O(1).$$
Hence,
$$S(T) = \sum \limits_{p \le T}A^{\Omega_{>z}(a(p))} \ll \begin{cases} T & \text{if $T \le z $,}\\ T(\log T/\log z)^{A-1}& \text{if $T > z $}.\end{cases}$$
Using Abel's summation and the fact that $1<c<2$,
$$\sum \limits_{p \le T}\frac{A^{\Omega_{>z}(p-1)}}{p^c} = \int_1^\infty t^{-c}dS(t) \le c\int_1^\infty \frac{S(t)}{t^{c+1}}dt \ll \int_1^z\frac{S(t)}{t^{c+1}}dt + \int_z^\infty\frac{S(t)}{t^{c+1}}dt$$
$$\ll \int_1^z \frac{1}{t}dt + \frac{1}{(\log z)^{A-1}}\int_z^\infty \frac{(\log t)^{A-1}}{t^c}dt.$$
We know $\int_1^z \frac{1}{t}dt = \log z$ and
$$\int_z^\infty \frac{(\log t)^{A-1}}{t^c}dt \le \int_1^\infty \frac{(\log t)^{A-1}}{t^c}dt.$$
Let $t = e^{u/(c-1)}$ and $dt = \frac{1}{(c-1)}e^{u/(c-1)}du$. Then,
$$\int_1^\infty \frac{(\log t)^{A-1}}{t^c}dt = \frac{1}{(c-1)^{A}}\int_1^\infty u^{A-1}e^{-u}du = \frac{\Gamma(A)}{(c-1)^{A}}.$$
Combining all the inequalities
$$\sum \limits_{n \le x}A^{\Omega_{>z}(a(n))} \le x \cdot z^{O(1)} \cdot \exp\left((c-1)\log x + O\left(\frac{\Gamma(A)}{(c-1)^{A}(\log z)^{A-1}} \right) \right).$$
Let 
$$c = 1+\frac{A}{(\log x)^{1/(A+1)} (\log z)^{(A-1)/(A+1)}}.$$
For this value of $c$,
$$(c-1)\log x + O\left(\frac{\Gamma(A)}{(c-1)^{A}(\log z)^{A-1}} \right) \ll {A}(\log x/\log z)^{1-\frac{1}{A+1}} (\log z)^{\frac{1}{A+1}} \ll \frac{\log x}{\exp((\log \log x)^{\eta/2})}.$$
\end{proof}
\noindent
Now we turn attention to the total number of prime factors of $a(n)$,
\begin{lemma}\label{Moment generating function B}
Fix $1 \le B < \sqrt[3]{2}$ and fix $a\in \{\phi,\sigma\}$. As $x \rightarrow \infty$,
$$\sum \limits_{n \le x}B^{\Omega(a(n))} \le x\exp(O_B((\log x)^{3/4})).$$
\end{lemma}
\begin{proof}
Let $c \in (1,2)$. By Rankin's trick,
$$\sum \limits_{n \le x}B^{\Omega(a(n))} \le x^c\sum \limits_{n \le x}\frac{B^{\Omega(a(n))}}{n^c} \le x^c \prod \limits_{p \le x}\left(1+\sum \limits_{k=1}^\infty\frac{B^{\Omega(a(p^k))}}{p^{kc}} \right)$$
$$\le x^c\exp \left(\sum \limits_{p\le x}\left[\frac{B^{\Omega(
a(p))}}{p^{c}} + \sum \limits_{k=2}^\infty\frac{B^{\Omega(
a(p^k))}}{p^{kc}}\right] \right).$$
Notice that $\Omega(a(p^k)) \le \frac{(k+1)\log p - \log(p-1)}{\log 2}$. Hence,
$$\sum \limits_{k=2}^\infty\frac{B^{\Omega(a(p^k))}}{p^{kc}} \le \frac{B^{\log(p)/\log(2)}}{B^{\log(p-1)/\log(2)}}\left( \sum \limits_{k=2}^\infty\frac{B^{k\log(p)/\log(2)}}{p^{kc}} \right).$$
Since $B < \sqrt[3]{2}<2$,
$$B^{\log(p)/\log(2)} = p^{\log B/\log 2} = p^{1 - (1-\log B/\log 2)}< p.$$
This implies
$$\sum \limits_{k=2}^\infty\frac{B^{k\log(p)/\log(2)}}{p^{kc}} \ll_B \frac{B^{2\log(p)/\log(2)}}{p^{2c}}$$
and,
$$x^c\exp \left(\sum \limits_{p\le x}\sum \limits_{k=1}^\infty\frac{B^{\Omega(a(p^k))}}{p^{kc}} \right) \le x\exp\left((c-1)\log x + O_B\left(\sum \limits_{p \le x}\left[\frac{B^{\Omega(
a(p))}}{p^{c}} + \frac{B^{\frac{3\log(p) - \log(p-1)}{\log(2)}}}{p^{2c}}\right] \right)\right).$$
From our choices of $B$ and $c$, 
$$3\log(B)/\log(2) - 2c < 3\log(\sqrt[3]{2})/\log(2) - 2 = -1.$$
Therefore,
$$\sum \limits_{p \le x}\frac{B^{\frac{3\log(p) - \log(p-1)}{\log(2)}}}{p^{2c}} \le \sum \limits_{p \le x}\frac{B^{\frac{3\log(p)}{\log(2)}}}{p^{2c}} = \sum \limits_{p \le x} p^{3\log(B)/\log(2) - 2c} = O(1).$$
Since $a(p) \in \{p-1,p+1\}$, $\sum \limits_{p \le T}{B^{\Omega(a(p))}} =\sum \limits_{p \le T}{B^{\Omega(a(p))}} \le \sum \limits_{d \le T+1}{2^{\Omega(d)}} \ll T(\log T)^2$ for all $T \ge 2$ (see \cite[Exercise 57, p. 59]{Ten}).\\
\\
Using partial summation as in the proof of Lemma \ref{Moment generating function A},
$$\sum \limits_{p \le x}\frac{B^{\Omega(a(p))}}{p^c} \ll \frac{1}{(c-1)^3}$$
so that
$$\sum \limits_{n \le x}B^{\Omega(a(n))} \le x\exp\left((c-1)\log x + O_B\left(\frac{1}{(c-1)^3} \right)\right).$$
Taking $c = 1+\frac{1}{(\log x)^{1/4}}$ completes the proof.
\end{proof}
\noindent
Similar to the proof in Lemma \ref{Moment generating function A}, it can be shown that $\sum \limits_{d \le T+1}B^{\Omega(d)} \ll T(\log T)^{B-1}$, which would allow $c = 1+ \frac{1}{(\log x)^{1/3}}$. Hence, Lemma \ref{Moment generating function B} can be improved as follows:
$$\sum \limits_{n \le x} B^{\Omega(a(n))} \le x\exp(O_B((\log x)^{2/3})).$$
\section{Proof of Theorem \ref{New Pomerance upper bound}}\label{Proof of Theorem 1.1}
This section will cover the proof of Theorem \ref{New Pomerance upper bound} by induction on $k$. For $k=1$, Theorem \ref{New Pomerance upper bound} follows from the sharper Theorem A. Now we assume the inductive hypothesis:
\begin{center}
As $n \rightarrow \infty$, for any $\beta < k-1$, $N_a^k(n) \le \frac{n}{L_{k,\beta+1}(n)^{1+o(1)}}$
\end{center}
and we prove 
\begin{center}
As $n \rightarrow \infty$, for any $\alpha < k$, $N_a^{k+1}(n) \le \frac{n}{L_{k+1,\alpha+1}(n)^{1+o(1)}}$.
\end{center}
To find an upper bound for the size of the preimages for the proof, we use that $\phi(n)$ and $\sigma(n)$ are within a small factor of $n$ (see \cite[pg 266-267]{Hardy}):
\begin{lemma}
There are positive constants $c_1$ and $c_2$ such that for all positive integers n,
$\sigma(n) \le c_1 n\log \log(3n)$ and $\phi(n) \ge c_2 n/\log \log(3n)$.
\end{lemma}
\noindent
Also, an estimate is needed from \cite{Normal} to bound the number of $m$ for which $a(m)$ is a multiple of a given integer $d$. 
\begin{lemma}
Let $a(n)$ be any of the functions $\phi(n)$ or $\sigma(n)$. Let $d$ be a positive integer and let $\ell = \Omega(d)$. For each $x \ge 1$, the number of $n \le x$ where $d \mid a(n)$ is at most
$$\frac{x}{d}(8\ell\log^2(ex))^\ell.$$
\end{lemma}
\noindent
It follows from Lemma 3 that for each fixed $k$ and all large enough positive integers $n$, any solution $m$ of the equation $a_k(m)=n$ satisfies $m \le n\log(2n)$. \\
\\
Thus, we set $x = n\log(2n)$.\\
\\
Fix $a \in \{\phi,\sigma\}$ and define $S = a_k^{-1}(n) = \{\ell:a_{k}(\ell) = n \}$. For each $\ell \in S$, let $a(m) = \ell$ and write $\ell = \ell'\ell''$, where $\ell'$ and $\ell''$ are the $z$-rough and $z$-smooth parts of $\ell$ respectively and $z = \exp((\log \log x)^{1/2})$ as defined in Section \ref{Bounds proof}. Below, $S$ is partitioned into 3 sets ($P$, $Q$ and $R$) by the number of $z$-rough prime factors and number of prime factors in total. Hence, if $a_{k+1}(m) = n$, then $a(m) \in S$, so that $m \in a^{-1}(P) \cup a^{-1}(Q) \cup a^{-1}(R)$. We note, as will be important below, that every element of $a^{-1}(P) \cup a^{-1}(Q) \cup a^{-1}(R)$ is at most $x$.\\
\\
\\
Our first set will contain the values of $\ell$ with many large prime factors. Let $\alpha < k$ and choose $\beta$ with $\alpha < \beta+1 < k$. Define $P =\{\ell \in S:\Omega_{>z}(\ell) \ge \log L_{k+1,\alpha}(x)$\}. If $a(m) \in P$, $\Omega_{>z}(a(m)) \ge \log L_{k+1,\alpha}(x)$. Moreover, by Lemma 1, 
$$\sum \limits_{ \overset{m \le x}{\Omega_{>z}(a(m)) \ge \log L_{k+1,\alpha}(x)}}1 \le \frac{1}{A^{\log L_{k+1,\alpha}(x)}}\sum \limits_{m \le x}A^{\Omega_{>z}(a(m))}$$
$$\le x\exp\left(\frac{\log x}{\exp(\log \log x)^{\eta/2}} -(1-\eta) \log L_{k+1,\alpha+1}(x)\right)= \frac{x}{L_{k+1,\alpha+1}(x)^{(1-\eta)+o(1)}}.$$
Since $\eta$ can be taken arbitrarily close to 0, we conclude that $\#a^{-1}(P) \le \frac{x}{L_{k+1,\alpha+1}(x)^{1+o(1)}}$. \\
    \\
    \\
Next, we focus on the second set containing the values of $\ell \notin P$ with a large number of total prime factors. Let $Q = \left\{\ell \in S\backslash P: \Omega(\ell) \ge {\log x}/{(\log \log x)^{k+\frac{1}{2}}} \right\}$. If $a(m) \in Q$, $\Omega(a(m)) \ge {\log x}/{(\log \log x)^{k+\frac{1}{2}}}$. By Lemma 2 setting $B = \frac{\sqrt[3]{2}+1}{2}$, 
$$\sum \limits_{ \overset{m \le x}{\Omega(a(m)) \ge \frac{\log x}{(\log \log x)^{k+\frac{1}{2}}}}}1 \le \frac{1}{B^{\frac{\log x}{(\log \log x)^{k+\frac{1}{2}}}}}\sum \limits_{m \le x}B^{\Omega(a(m))}$$
$$\le x\exp\left(C_B(\log x)^{3/4} - \frac{\log x}{(\log \log x)^{k+\frac{1}{2}}}\log(B) \right) \le x \exp\left((-D+o(1))\frac{\log x}{(\log \log x)^{k+\frac{1}{2}}}\right)$$
for some positive constant $D$.\\
Hence, an upper bound for $\#a^{-1}(Q)$ is $x \exp\left((-D+o(1))\frac{\log x}{(\log \log x)^{k+\frac{1}{2}}}\right)$.\\
\\
\\
Our final set $R$ will contain everything else in $S$: $R = S\backslash (P \cup Q)$. Given $\ell$, fix $\ell'$. Let $j = \Omega(\ell') = \Omega_{>z}(\ell) < \log L_{k+1,\alpha}(x)$. Since $\ell' \mid \ell = a(m)$, we can use Lemma 4 to find an upper bound on the number of $m$ corresponding to a fixed $\ell'$. The upper bound is given as follows:
    $$\frac{x}{\ell'}(8j \log^2 (ex))^j \le \frac{x}{\ell'}\exp\left((3+o(1))\log L_{k,\alpha}(x)\right) = \frac{xL_{k,\alpha}(x)^{3+o(1)}}{\ell'}.$$
    Since $\Omega(\ell) < \frac{\log x}{(\log \log x)^{k+\frac{1}{2}}}$, $\ell'' \le z^{\Omega(\ell)} < \exp(\log x/ (\log \log x)^k)= L_{k,\alpha}(x)^{o(1)}$. Summing over all $\ell'$ to obtain an upper bound on $\#a^{-1}(R)$, we get
    $$xL_{k,\alpha}(x)^{3+o(1)}\sum \limits_{\ell \in R}\frac{1}{\ell'}= xL_{k,\alpha}(x)^{3+o(1)}\sum \limits_{\ell \in R}\frac{\ell''}{\ell}\le xL_{k,\alpha}(x)^{3+o(1)}\sum \limits_{\ell \in S}\frac{1}{\ell}.$$
Since  $a(\ell) = n$, by Lemma 3, $\ell \ge c\frac{n}{\log \log 3n}$ for a certain absolute positive constant $c$, and so by our inductive hypothesis, $\sum \limits_{\ell \in S}\frac{1}{\ell} \le \frac{c\log \log 3n}{n}\sum \limits_{\ell \in S}1 \le L_{k,\beta+1}(n)^{-1+o(1)} = L_{k,\beta+1}(x)^{-1+o(1)}$.
Therefore, from our choice of $\alpha$,
        $$\#a^{-1}(R) \le x\exp\bigg((3+o(1))\log L_{k,\alpha}(x) -(1+o(1))\log L_{k,\beta+1}(x)\bigg)  = \frac{x}{L_{k,\beta+1}(x)^{1+o(1)}}.$$
We add all the upper bounds together to get an upper bound for $N_a^{k+1}(n) = \#a^{-1}(S)$. We will get
$$N_a^{k+1}(n) = \#a^{-1}(S)  =\#a^{-1}(P)+\#a^{-1}(Q)+\#a^{-1}(R)$$
$$\le \frac{x}{L_{k+1,\alpha+1}(x)^{1+o(1)}} + x\exp\left((-D+o(1))\frac{\log x}{(\log \log x)^{k+\frac{1}{2}}}\right) + \frac{x}{L_{k,\beta+1}(x)^{1+o(1)}} $$
$$\le \frac{x}{L_{k+1,\alpha+1}(x)^{1+o(1)}} = \frac{n}{L_{k+1,\alpha+1}(n)^{1+o(1)}},$$
thus completing the proof of Theorem \ref{New Pomerance upper bound}.
\section{Remarks}\label{Remarks for lower bound}
The proof can be extended to any order of composition of $\phi$ and $\sigma$:\\
\\
Let $a^{(1)}, \hdots, a^{(k)} \in \{\phi, \sigma\}$ for a fixed $k \in \mathbb{N}$, and define $A = a^{(1)} \circ \cdots \circ a^{(k)}$. Then for fixed $\beta<k-1$, as $n \rightarrow \infty$, 
we have
$$\#\{m: A(m) = n\} \le \frac{n}{L_{k,\beta+1}(n)^{1+o(1)}}.$$
It is interesting to speculate about the sharp maximal order of $\#\phi_k^{-1}(m)$. When $k=1$, we have noted already Pomerance's conjecture that there is a sequence of $m$ tending to $\infty$ along which $\#\phi^{-1}(m) = m/L(m)^{1+o(1)}$. One reason for believing this conjecture is that it becomes a theorem under the following plausible hypothesis on smooth shifted primes $p-1$:
\begin{hypothesis}
Let $\Psi(x,y)$ denote the number of integers $n \le x$ free of prime factors exceeding $y$ and $\Pi(x,y)$ denote the number of primes $p \le x$ such that $p-1$ is free of prime factors exceeding $y$. If $x \ge y$ and $y \rightarrow \infty$,
$$\frac{1}{x}\Psi(x,y) \sim \frac{1}{\pi(x)}\Pi(x,y).$$
\end{hypothesis}
\noindent
Under this hypothesis, Pomerance showed \cite{Pidgeon} that there is a set of $n \le x$ of size $x/L(x)^{1+o(1)}$ (as $x\rightarrow \infty$) for which all the values $\phi(n)$ are $\log x$ smooth. Since the cardinality of $\log x$ smooth numbers up to $x$ is only $L(x)^{o(1)}$, it follows that some $(\log x)$-smooth number $m \le x$ has at least $x/L(x)^{1+o(1)}$ preimages $n \le x$. Hence, $\#\phi^{-1}(m) \ge x/L(x)^{1+o(1)} \ge m/L(m)^{1+o(1)}$.\\
\\
Lamzouri has studied smooth values of the iterates of $\phi$. He proves in \cite{Lam}, conditional on the Elliott--Halberstam conjecture, that:
\renewcommand*{\thetheorem}{\arabic{theorem}}
\setcounter{theorem}{1}
\begin{theorem}
Define $\Phi_k(x,y) = \lvert \{m \le x: p\mid \phi_k(m) \Rightarrow p \le y\} \rvert$. Fix $U>1$. If $y = x^{1/u}$ where $1 \le u \le U$, then as $x \rightarrow \infty$,
$$\Phi_k(x,y) \sim x\rho_k(u)$$
where $\rho_0(u) = \left(\frac{e + o(1)}{u \log u}\right)^u$ and when $k \ge 1$,
$$\rho_k(u) = \left(\frac{1+o(1)}{\log_k(u) \log_{k+1}(u)} \right)^u$$
as $u \rightarrow \infty$, where $\log_k(u) = \underbrace{\log \log \cdots \log}_{\text{$k$ times}}u$.
\end{theorem}
\noindent
In Lamzouri's conjecture, the smoothness parameter $y$ exceeds a fixed positive power of $x$. It is tempting to conjecture, optimistically, that Lamzouri's conjecture remains valid down to $y = \log x$. If so, we would derive (analogously to Pomerance) $\#\phi_k^{-1}(m) \ge m\exp\left((-1+o(1))\frac{\log m \log_{k+2} m}{\log_2 m} \right)$ on a sequence of $m$ tending to infinity. If true, it is also tempting to conjecture that this is sharp.

\bibliographystyle{unsrt}

\newpage
\bibliography{Dissertation/ReferencesF}

\end{document}